\documentclass[twoside]{article}
\usepackage[cp1251]{inputenc}
\usepackage[T2A]{fontenc}
\usepackage{array}
\usepackage{amssymb}
\usepackage{amsmath}
\usepackage{amsthm}
\usepackage{latexsym}
\usepackage{indentfirst}
\usepackage{bm}
\usepackage{enumerate}
\usepackage[dvips]{graphicx}
\usepackage{epsf}
\usepackage{wrapfig}
\usepackage{euscript}
\usepackage{indentfirst}
\usepackage[english,russian]{babel}
\usepackage{textcomp}
\newtheorem{theorem}{Theorem}
\newtheorem{lemma}{Lemma}
\newtheorem{definition}{Definition}

\newtheorem{corollary}{Corollary}
\newcommand{\co}{{\hspace{0.25mm}:\hspace{0.25mm}}}
\begin{document}
{\selectlanguage{english}
\binoppenalty = 10000 %
\relpenalty   = 10000 %

\pagestyle{headings} \makeatletter
\renewcommand{\@evenhead}{\raisebox{0pt}[\headheight][0pt]{\vbox{\hbox to\textwidth{\thepage\hfill \strut {\small Grigory. K. Olkhovikov}}\hrule}}}
\renewcommand{\@oddhead}{\raisebox{0pt}[\headheight][0pt]{\vbox{\hbox to\textwidth{{Completeness for explicit jstit logic}\hfill \strut\thepage}\hrule}}}
\makeatother

\title{Justification announcements in discrete time. Part II: frame definability results}
\author{Grigory K. Olkhovikov\\Ruhr University Bochum\\
Department of Philosophy II; NAFO 02/299\\
Universit\"{a}tstr. 150, D-44780, Bochum, Germany\\
email: grigory.olkhovikov@rub.de, grigory.olkhovikov@gmail.com}
\date{}
\maketitle

\begin{quote}
\textbf{Abstract}. In Part I of this paper, we presented a
Hilbert-style system $\Sigma_D$ axiomatizing of stit logic of
justification announcements (JA-STIT) interpreted over models with
discrete time structure. In this part, we prove three frame
definability results for $\Sigma_D$ using three different
definitions of a frame plus a yet another version of completeness
result.
\end{quote}

\begin{quote}
stit logic, justification logic, strong completeness, frame
definability
\end{quote}

\section{Introduction}

The so-called stit logic of justification announcements (JA-STIT,
for short) was introduced in \cite{OLexpl-gen} as an explicit
fragment of the richer logic of $E$-notions introduced in
\cite{OLWA2}. The underlying idea was to interpret the proving
activity as an activity that results in presenting (or, as it
were, \emph{demonstrating}) proofs to the community thus making
them publicly available within this community. JA-STIT borrows the
representation of proofs which get presented to the community in
this way from justification logic by S. Artemov et al.
\cite{ArtemovN05}, whereas the model for the agentive activities
within the community is based on stit logic by N. Belnap et al.
\cite{belnap2001facing}. Both of these logics are imported into
JA-STIT rather explicitly, which leads to the presence of a full
set of their respective modalities in the JA-STIT language. In a
similar fashion, the intended models for JA-STIT contain the full
set of structural elements present in the models for both
justification logic and stit logic.

The interaction between agents and proofs is then provided for by
a common pool of proofs publicly announced within the community.
This model for justification announcements is inspired by a rather
common occurrence when a group of agents tries to produce a proof
working on a shared whiteboard. In JA-STIT this situation is
modelled in an idealized form, so that one abstracts away from (1)
the other available media (like private notes, private messages,
etc.), (2) the natural limitations of the actual whiteboard (like
limited space and the necessity to erase old proofs), and (3) from
the natural limitations of the agents' communication capacities
(like bad handwriting on the part of presenting agents or
short-sightedness on the part of spectators).

The state of the common body of publicly presented proofs, or of
the community whiteboard, as we will sometimes call it, is
described in JA-STIT by modality $Et$, where $t$ is an arbitrary
proof polynomial of justification logic. The informal
interpretation of $Et$ is that the proof $t$ is presented to the
community, or that $t$ is on the whiteboard. This reading also
explains the choice of $E$ as notation for this modality, since it
serves as a sort of existence predicate for the pool of proofs
publicly announced within the community.

The axiomatization of JA-STIT w.r.t. the full class of its
intended models was given in \cite{OLexpl-gen} in the form of
Hilbert-style axiomatic system $\Sigma$. At the same time,
Proposition 1 of \cite{OLexpl-gen} showed that, rather
surprisingly, this axiomatization is sensitive to the temporal
structure of the underlying models, even though neither stit
logic, nor justification logic, nor else $Et$-modality seem to be
relevant to temporal logic, and the standard temporal modalities
prove undefinable within JA-STIT. Nevertheless, it turned out that
once the class of underlying models is restricted to the models
based on discrete time, the axiomatization is no longer complete.
The first part of this paper focuses on finding a strongly
complete axiomatization $\Sigma_D$ of JA-STIT over the subclass of
its intended models which are based on discrete time.  We also
found a number of less restrictive classes of models in the
process --- which all induce the same set of validities as the
models with discrete temporal substructure. This result shows that
one cannot \textbf{enforce} a discrete temporal substructure onto
a model by simply postulating an appropriate set of JA-STIT
validities. A natural follow-up question then presents itself,
namely, how much of a structure can be enforced on the underlying
model by simply postulating the set of theorems of $\Sigma_D$.
Given that JA-STIT is a variant of modal propositional logic, it
is more productive to recast this question in terms of frame
definability rather than model definability. In this way, we ask:
\begin{quote}
\textbf{Main question}. Assuming all the theorems of $\Sigma_D$
are valid over the class of models based on a given frame $F$,
what can be said about $F$ itself?
\end{quote}

The exact meaning of this question clearly depends on how we
define the notion of a frame on which a given model is based.
Indeed, if we are primarily interested in what our axiomatization
has to say about temporal sub-structure of the underlying frame,
we need to include at least the set of moments woven together by a
temporal accessibility relation. This gives us what we call a
temporal frame; but we will show below that the restriction on the
class of frames induced by the set of $\Sigma_D$ theorems is not
affected at all if we also add the choice function to the frame
structure thus extending a temporal frame to a stit frame. By
contrast, the situation changes dramatically if we further add to
a stit frame the epistemic accessibility relations as these can
interact with the stit substructure of the frame in most intricate
and subtle ways. In this way we obtain a justification stit frame
and another frame definability theorem which is very different
from the similar results for temporal and stit frames.

The layout of the rest of the paper is as follows. In Section
\ref{basic} we define the language and the semantics of the logic
at hand. We then define the three versions of the frame notion
mentioned in the previous paragraph and consider some natural
subclasses in both frame types relevant to the main part of the
paper (these will appear in the results presented in Section
\ref{main}). Next, we recall the definition of $\Sigma_D$ and
recapitulate, without a proof, some results from Part I to be used
in this second part as well.

Section \ref{main} proves the frame definability results for the
three versions of a frame notion. Additionally, we identify yet
another class of models w.r.t. which our axiomatization is
complete, this time using the notion of justification stit frame.
Section \ref{conclusion} gives some conclusions and drafts
directions for future work.

In what follows we will be assuming, due to space limitations, a
basic acquaintance with both stit logic and justification logic.
We refer the readers to \cite{sep-logic-justification} for a quick
introduction into the basics of stit logic, and \cite[Ch.
2]{horty2001agency} for the same w.r.t. justification logic.

\section{Preliminaries}\label{basic}

\subsection{Basic definitions and notation}
We fix some preliminaries. First we choose a finite set $Ag$
disjoint from all the other sets to be defined below. Individual
agents from this set will be denoted by letters $i$ and $j$. Then
we fix countably infinite sets $PVar$ of proof variables (denoted
by $x,y,z$) and $PConst$ of proof constants (denoted by $c,d$).
When needed, subscripts and superscripts will be used with the
above notations or any other notations to be introduced in this
paper. Set $Pol$ of proof polynomials is then defined by the
following BNF:
$$
t := x \mid c \mid s + t \mid s \times t \mid !t,
$$
with $x \in PVar$, $c \in PConst$, and $s,t$ ranging over elements
of $Pol$. In the above definition $+$ stands for the \emph{sum} of
proofs, $\times$ denotes \emph{application} of its left argument
to the right one, and $!$ denotes the so-called
\emph{proof-checker}, so that $!t$ checks the correctness of proof
$t$.

In order to define the set $Form^{Ag}$ of formulas we fix a
countably infinite set $Var$ of propositional variables to be
denoted by letters $p,q$. Formulas themselves will be denoted by
letters $A,B,C,D$, and the definition of $Form^{Ag}$ is supplied
by the following BNF:
\begin{align*}
A := p \mid A \wedge B \mid \neg A \mid [j]A \mid \Box A \mid t\co
A \mid KA \mid Et,
\end{align*}
with $p \in Var$, $j \in Ag$ and $t \in Pol$.

It is clear from the above definition of $Form^{Ag}$ that we are
considering a version of modal propositional language. As for the
informal interpretations of modalities, $[j]A$ is the so-called
cstit action modality and $\Box$ is the historical necessity
modality, both modalities are borrowed from stit logic. The next
two modalities, $KA$ and $t\co A$, come from justification logic
and the latter is interpreted as ``$t$ proves $A$'', whereas the
former is the strong epistemic modality ``$A$ is known''.

We assume $\Diamond$ as notation for the dual modality of $\Box$.
As usual, $\omega$ will denote the set of natural numbers
including $0$, ordered in the natural way.

For the language at hand, we assume the following semantics. A
justification stit (jstit) model for $Ag$ is a structure
$$
\mathcal{M} = \langle Tree, \unlhd, Choice, Act, R, R_e,
\mathcal{E}, V\rangle
$$
such that:
\begin{enumerate}
\item $Tree$ is a non-empty set. Elements of $Tree$ are called
\emph{moments}.

\item $\unlhd$ is a partial order on $Tree$ for which a temporal
interpretation is assumed. We will also freely use notations like
$\unrhd$, $\lhd$, and $\rhd$ to denote the inverse relation and
the irreflexive companions.\footnote{A more common notation $\leq$
is not convenient for us since we also widely use $\leq$ in this
paper to denote the natural order relation between elements of
$\omega$.}

\item $Hist$ is a set of maximal chains in $Tree$ w.r.t. $\unlhd$.
Since $Hist$ is completely determined by $Tree$ and $\unlhd$, it
is not included into the structure of model as a separate
component. Elements of $Hist$ are called \emph{histories}. The set
of histories containing a given moment $m$ will be denoted $H_m$.
The following set:
$$
MH(\mathcal{M}) = \{ (m,h)\mid m \in Tree,\, h \in H_m \},
$$
called the set of \emph{moment-history pairs}, will be used to
evaluate formulas of the above language.

\item $Choice$ is a function mapping $Tree \times Ag$ into
$2^{2^{Hist}}$ in such a way that for any given $j \in Ag$ and $m
\in Tree$ we have as $Choice(m,j)$ (to be denoted as $Choice^m_j$
below) a partition of $H_m$. For a given $h \in H_m$ we will
denote by $Choice^m_j(h)$ the element of partition $Choice^m_j$
containing $h$.

\item $Act$ is a function mapping $MH(\mathcal{M})$ into
$2^{Pol}$.

\item $R$ and $R_e$ are two pre-order on $Tree$ giving two
versions of epistemic accessibility relation. They are assumed to
be connected by the inclusion $R \subseteq R_e$.

\item $\mathcal{E}$ is a function mapping $Tree \times Pol$ into
$2^{Form}$ called \emph{admissible evidence function}.

\item $V$ is an evaluation function, mapping the set $Var$ into
$2^{MH(\mathcal{M})}$.
\end{enumerate}

Furthermore, a jstit model has to satisfy a number of additional
constraints. In order to facilitate their exposition, we introduce
a couple of useful notations first. For a given $m \in Tree$ and
any given $h, g \in H_m$ we stipulate that:
$$
Act_m := \bigcap_{h \in H_m}(Act(m,h),
$$
and:
$$
h \approx_m g \Leftrightarrow (\exists m' \rhd m)(h, g \in
H_{m'}).
$$
Whenever we have $h \approx_m g$, we say that $h$ and $g$ are
\emph{undivided} at $m$.

The list of constraints on jstit models then looks as follows:

\begin{enumerate}
\item Historical connection:
$$
(\forall m,m_1 \in Tree)(\exists m_2 \in Tree)(m_2 \unlhd m \wedge
m_2 \unlhd m_1).
$$

\item No backward branching:
$$
(\forall m,m_1,m_2 \in Tree)((m_1 \unlhd m \wedge m_2 \unlhd m)
\Rightarrow (m_1 \unlhd m_2 \vee m_2 \unlhd m_1)).
$$

\item No choice between undivided histories:
$$
(\forall m \in Tree)(\forall h,h' \in H_m)(h \approx_m h'
\Rightarrow Choice^m_j(h) = Choice^m_j(h'))
$$
for every $j \in Ag$.

\item Independence of agents:
$$
(\forall m\in Tree)(\forall f:Ag \to 2^{H_m})((\forall j \in
Ag)(f(j) \in Choice^m_j) \Rightarrow \bigcap_{j \in Ag}f(j) \neq
\emptyset).
$$

\item Monotonicity of evidence:
$$
(\forall t \in Pol)(\forall m,m' \in Tree)(R_e(m,m') \Rightarrow
\mathcal{E}(m,t) \subseteq \mathcal{E}(m',t)).
$$

\item Evidence closure properties. For arbitrary $m \in Tree$,
$s,t \in Pol$ and $A, B \in Form$ it is assumed that:
\begin{enumerate}
\item $A \to B \in \mathcal{E}(m,s) \wedge A \in \mathcal{E}(m,t)
\Rightarrow B \in \mathcal{E}(m,s\times t)$;

\item $\mathcal{E}(m,s) \cup \mathcal{E}(m,t) \subseteq
\mathcal{E}(m,s + t)$. \item $A \in \mathcal{E}(m,t) \Rightarrow
t:A \in \mathcal{E}(m,!t)$;
\end{enumerate}

\item Expansion of presented proofs:
$$
(\forall m,m' \in Tree)(m' \lhd m \Rightarrow \forall h \in H_m
(Act(m',h) \subseteq Act(m,h))).
$$

\item No new proofs guaranteed:
$$
(\forall m \in Tree)(Act_m \subseteq \bigcup_{m' \lhd m, h \in
H_m}(Act(m',h))).
$$

\item Presenting a new proof makes histories divide:
$$
(\forall m \in Tree)(\forall h,h' \in H_m)(h \approx_m h'
\Rightarrow (Act(m,h) = Act(m,h'))).
$$

\item Future always matters:
$$
\unlhd \subseteq R.
$$

\item Presented proofs are epistemically transparent:
$$
(\forall m,m' \in Tree)(R_e(m,m') \Rightarrow (Act_m \subseteq
Act_{m'})).
$$
\end{enumerate}
The components like $Tree$, $\unlhd$, $Choice$ and $V$ are
inherited from stit logic, whereas $R$, $R_e$ and $\mathcal{E}$
come from justification logic. The only new component is $Act$
which represents the above-mentioned common pool of proofs
demonstrated to the community or the state of the community
whiteboard at any given moment under a given history. When
interpreting $Act$, we invoke the classical stit distinction
between dynamic (agentive) and static (moment-determinate)
entities, assuming that the presence of a given proof polynomial
$t$ on the community whiteboard only becomes an accomplished fact
at $m$ when $t$ is present in $Act(m,h)$ for \emph{every} $h \in
H_m$. On the other hand, if $t$ is in $Act(m,h)$ only for
\emph{some} $h \in H_m$ this means that $t$ is rather in a dynamic
state of \emph{being presented}, rather than being present, to the
community.

Due to space limitations, we skip the explanation of the
intuitions behind jstit models. The interested reader may find
such an explanation either in Section 2 of Part I of this paper,
or in \cite[Section 3]{OLWA}.

For the members of $Form^{Ag}$, we will assume the following
inductively defined satisfaction relation. For every jstit model
$\mathcal{M} = \langle Tree, \unlhd, Choice, Act, R, R_e,
\mathcal{E}, V\rangle$ and for every $(m,h) \in MH(\mathcal{M})$
we stipulate that:
\begin{align*}
&\mathcal{M}, m, h \models p \Leftrightarrow (m,h) \in
V(p);\\
&\mathcal{M}, m, h \models [j]A \Leftrightarrow (\forall h'
\in Choice^m_j(h))(\mathcal{M}, m, h' \models A);\\
&\mathcal{M}, m, h \models \Box A \Leftrightarrow (\forall h'
\in H_m)(\mathcal{M}, m, h' \models A);\\
&\mathcal{M}, m, h \models KA \Leftrightarrow \forall m'\forall
h'(R(m,m') \& h' \in H_{m'} \Rightarrow \mathcal{M}, m', h'
\models A);\\
&\mathcal{M}, m, h \models t\co A \Leftrightarrow A \in
\mathcal{E}(m,t) \& (\forall m' \in Tree)(R_e(m,m') \& h' \in
H_{m'} \Rightarrow \mathcal{M}, m', h'
\models A);\\
&\mathcal{M}, m, h \models Et \Leftrightarrow t \in Act(m,h).
\end{align*}
In the above clauses we assume that $p \in Var$; we also assume
standard clauses for Boolean connectives.

\subsection{Frames and their subclasses}

In a modal propositional context, it is customary to consider
frames alongside models, and frames are normally defined as
reducts of the models to components not involving any linguistic
entities. In this way, within the classical modal logic a frame is
a model minus the evaluation for propositional variables. Thus, in
stit logic, a frame will contain $Tree$, $\unlhd$ and $Choice$ but
omit $V$. In pure justification logic the situation is slightly
more complicated, since also the admissible evidence function
$\mathcal{E}$ invokes polynomials and sets of formulas. Therefore,
in \cite{ArtemovN05} a justification frame is just a set of worlds
pre-ordered by the two epistemic accessibility relations; it does
not contain $\mathcal{E}$ which will rather be construed as a part
of a model based on this frame. When we turn to jstit models, we
find $Act$ as a further language-dependent component. Even though
one can argue that with $Act$ we enter a sort of grey area as
compared to $\mathcal{E}$, since $Act$ invokes proof polynomials
but not formulas, in the context of JA-STIT it is clear that $Act$
must be outside of frame structure. Indeed, one of the traditional
distinctions between frames and models would be that one can
evaluate formulas of a given language in models but not in frames.
An exception is made for $0$-ary connectives like $\bot$ and
$\top$ and the formulas built from these connectives. Now, within
the context of JA-STIT one can evaluate every formula of the form
$Et$ for $t \in Pol$ using $Act$ alone, and it would be tough to
argue that such formulas are just another example of $0$-ary
connectives.

Having these considerations in mind, we define our frame notions
as follows. If $\mathcal{M} = \langle Tree, \unlhd, Choice, Act,
R, R_e, \mathcal{E}, V\rangle$ is a jstit model for $Ag$, then $F
= \langle Tree, \unlhd, Choice, R, R_e\rangle$ is a
\emph{justification stit} (or jstit, for short) \emph{frame} for
$Ag$, and $\mathcal{M}$ is said to be based on $F$. Similarly, we
define that $C = \langle Tree, \unlhd, Choice\rangle$ (resp. $T =
\langle Tree, \unlhd\rangle$) is a \emph{stit frame} (resp.
\emph{temporal} frame) for $Ag$. In this case, both $\mathcal{M}$
and $F$ as defined above are said to be based on $C$ (resp. $T$).
Given a class $\mathcal{F}$ of jstit (resp. stit, temporal)
frames, we will denote the class of jstit models based on the
frames from $\mathcal{F}$ by $Mod(\mathcal{F})$. We also observe,
that notations like $Hist(F)$ and $MH(F)$ still make perfect sense
when $F$ is a jstit, stit, or temporal frame.

An important subsclass of jstit frames is made up of what we will
call \emph{unirelational} jstit frames. These are the frames
satisfying the additional constraint that $R_e \subseteq R$. It is
known (see \cite{ArtemovN05}) that switching from the full class
of models to the unirelational models (that is to say, to the
models based on unirelational frames) in the context of pure
justification logic still leaves us with a class of models w.r.t.
which the logic is complete. We have shown (in \cite{OLexpl-gen})
that the same holds for JA-STIT over the general class of models
and (in Part I of this paper) that the situation does not change
when one restricts attention to the models based on any class of
stit frames considered in Part I of this paper. In this second
part, we will also show that this observation holds good for the
class of models based on regular jstit frames. Whenever
$\mathcal{F}$ is a class of jstit frames, we will denote $\{ F \in
\mathcal{F}\mid F\textup{ is unirelational} \}$ by
$\mathcal{F}\downarrow$. Similarly, whenever $\mathcal{C}$ is a
class of stit or temporal frames we will denote by
$Mod^{\downarrow}(\mathcal{C})$ the class of unirelational jstit
models based on frames from $\mathcal{C}$.

Before we move on, we need to introduce the notation for an
immediate $\lhd$-successor of a given moment as it will play an
important part in the frame restrictions to be considered below.
So whenever $C = \langle Tree, \unlhd, Choice\rangle$ is a stit
frame and $m, m' \in Tree$, we set that:
$$
Next(m,m') \Leftrightarrow (m \lhd m' \& (\forall m'' \lhd m')(m''
\unlhd m)).
$$

We now remind the reader the definition of a mixed successor stit
frame originally given in Part I:

\begin{definition}\label{classes-stit}
Let $C = \langle Tree, \unlhd, Choice\rangle$ be a stit frame.
Then we say that $C$ is a \textbf{mixed successor} frame iff for
all $m,m_1 \in Tree$ it is true that:
\begin{align}
&[m \lhd m_1 \Rightarrow (\exists m_2 \unlhd m_1)(Next(m, m_2))]
\vee [(\forall h,g \in H_m)(h \approx_m
g)]\label{mixsucc}\tag{\text{mixsucc}}
\end{align}
\end{definition}

We will denote the class of  mixed successor stit frames by
$\mathcal{C}_{mixsucc}$, and the same class restricted to the stit
frames for a given community $Ag$ by $\mathcal{C}^{Ag}_{mixsucc}$.
Since condition \eqref{mixsucc} does not invoke $Choice$ function
of a stit frame, it makes perfect sense for temporal frames as
well. Therefore, we will denote by $\mathcal{T}_{mixsucc}$ the
class of mixed successor temporal frames and by
$\mathcal{T}^{Ag}_{mixsucc}$ the class of such frames for $Ag$.

We now proceed to defining the restriction used in our jstit frame
definability result. First we need one further technical notion:

\begin{definition}\label{sigma}
Let $F = \langle Tree, \unlhd, Choice, R, R_e\rangle$ be a jstit
frame and let $m \in Tree$. We define $\Theta_m \subseteq
2^{2^{Tree}}$ setting that $S \subseteq Tree$ is in $\Theta_m$ iff
all of the following conditions hold:
\begin{enumerate}
\item $m \in S$;

\item $(\forall m_1,m_2 \in Tree)((m_1 \in S \& R_e(m_1,m_2))
\Rightarrow m_2 \in S)$;

\item $(\forall m_1 \in Tree)[(\forall h \in H_{m_1})(\exists m_2
\in h)(Next(m_1,m_2)\& m_2 \in S) \Rightarrow m_1 \in S]$;

\item $(\forall m_1 \in Tree)([m_1 \in S \& (\forall m_2 \lhd
m_1)\exists m_3(m_2 \lhd m_3 \lhd m_1)] \Rightarrow (\exists m_4
\lhd m_1)(m_4 \in S))$.
\end{enumerate}
\end{definition}

We give one important consequence of the above definition as a
lemma:
\begin{lemma}\label{theta}
Let $F = \langle Tree, \unlhd, Choice, R, R_e\rangle$ be a jstit
frame, let $m \in Tree$, and let $S \in \Theta_m$. Then:
$$
(\forall m_0 \in Tree)(m_0 \in S \Rightarrow (\exists m_1 \in
Tree)(m_1 \lhd m_0)).
$$
\end{lemma}
\begin{proof}
Assume, for contradiction, that $m_0 \in S$ but there is no moment
$m_1$ such that $m_1 \lhd m_0$. Then, by contraposition of
Definition \ref{sigma}.4, we must have that:
$$
\neg(\forall m_2 \lhd m_0)\exists m_3(m_2 \lhd m_3 \lhd m_0),
$$
whence, pushing the negation inside, we get that:
$$
(\exists m_2 \lhd m_0)\forall m_3(m_2 \lhd m_3 \Rightarrow \neg
m_3 \lhd m_0).
$$
In particular, for any such $m_2$ we will have $m_2 \lhd m_0$ and
thus we have got our contradiction in place.
\end{proof}
Lemma \ref{theta} shows that for a given $m \in Tree$ the family
$\Theta_m$ may end up being empty, for example, when we have
$R_e(m,m')$ and $m'$ is the $\unlhd$-least moment in $Tree$. On
the other hand, in case when $Tree$ has no $\unlhd$-least moment,
$\Theta_m$ is always non-empty, since we will have $Tree \in
\Theta_m$ for all moments $m$. However, within this paper we will
be mostly interested in less trivial configurations of $\Theta_m$
families:

\begin{definition}\label{classes-jstit}
Let $F = \langle Tree, \unlhd, Choice, R, R_e\rangle$ be a jstit
frame. Then we say that $F$ is \textbf{regular} iff the following
holds for all $m, m_1 \in Tree$:
\begin{align}
&\{m \lhd m_1 \& (\exists S \in \bigcap_{m \lhd m_0 \unlhd
m_1}\Theta_{m_0})(m \notin S \&\notag\\
&\&(\exists h' \in H_m)((\forall g \in H_{m_1})(h' \not\approx_m g) \& (\forall m' \in h')[Next(m,m') \Rightarrow m' \notin S])\} \Rightarrow\notag\\
&\quad\qquad\qquad\qquad\qquad\qquad\qquad\qquad\qquad\Rightarrow
(\exists m_2 \unlhd m_1)(Next(m,m_2))\label{reg}\tag{\text{reg}}
\end{align}
\end{definition}
Just as with the \eqref{mixsucc} restriction on stit frames, we
introduce the notation $\mathcal{F}_{reg}$ for the class of
regular jstit frames and the notation $\mathcal{F}^{Ag}_{reg}$ for
the class of regular jstit frames for a given agent community
$Ag$.

Before we move on to actually proving something, we mention a
couple of technical lemmas which were proved in Part I and will be
used here without a proof:
\begin{lemma}\label{technical}
Let $C = \langle Tree, \unlhd, Choice \rangle$ be a stit frame.
Then:

\begin{enumerate}
\item $(\forall m \in Tree)(\forall h \in H_m)(\exists m'(m' \rhd
m) \Rightarrow (\exists m'' \rhd m)(h \in H_{m''}))$;

\item $(\forall m, m' \in Tree)(m \unlhd m' \Rightarrow H_{m'}
\subseteq H_m)$;

\item $\approx_m$ is an equivalence relation for every $m \in
Tree$.
\end{enumerate}
\end{lemma}
\begin{lemma}\label{technical-new}
Let $\mathcal{M} = \langle Tree, \unlhd, Choice, Act, R, R_e,
\mathcal{E}, V \rangle$ be a jstit model. Then:
$$
(\forall m,m' \in Tree)(\forall h \in H_{m'})(\forall t \in Pol)(m
\lhd m' \& t\in Act(m,h') \Rightarrow t \in Act_{m'}).
$$
\end{lemma}

We now establish a connection between the above-defined classes of
stit and jstit frames:

\begin{lemma}\label{stit-jstit}
Let $Ag$ be a community of agents, let $F = \langle Tree, \unlhd,
Choice, Act, R,R_e\rangle$ be a jstit frame for $Ag$, and let $C$
be its reduct to stit frame. Then all of the following statements
are true:
\begin{enumerate}
\item If $C \in \mathcal{C}^{Ag}_{mixsucc}$, then $F \in
\mathcal{F}^{Ag}_{reg}$;

\item It is possible that $F \in \mathcal{F}^{Ag}_{reg}$ but $C
\notin \mathcal{C}^{Ag}_{mixsucc}$.
\end{enumerate}
\end{lemma}
\begin{proof}
(Part 1) Assume that $C \in \mathcal{C}^{Ag}_{mixsucc}$; we show
that $F \in \mathcal{F}^{Ag}_{reg}$. Indeed, assume that $m, m_1
\in Tree$, $h' \in H_m$, and $S \subseteq Tree$ verify the
antecedent of \eqref{reg}. This implies, among other things that:
\begin{equation}\label{E:y1}
    m \lhd m_1 \& (\forall g \in H_{m_1})(h' \not\approx_m g).
\end{equation}
Now, choose any $h \in H_{m_1}$. By Lemma \ref{technical}.2, we
get that $h \in H_m$, and, by the second conjunct of \eqref{E:y1},
we obtain that $h \not\approx_m h'$ thus falsifying the second
disjunct in the condition \eqref{mixsucc} for $m, m_1$. Therefore,
the first disjunct of the same condition must hold, whereby, given
that $m \lhd m_1$, we get that $(\exists m_2 \unlhd
m_1)(Next(m,m_2))$, as desired.

As for Part 2, consider a jstit frame $F = \langle Tree, \unlhd,
Choice, Act, R,R_e\rangle$ such that $C$ is outside
$\mathcal{C}^{Ag}_{mixsucc}$ and $R = R_e = Tree \times Tree$.
With these settings we will have:
\begin{align*}
\Theta_m =\left\{%
\begin{array}{ll}
    \emptyset, & \hbox{if there is a $\unlhd$-least element in $Tree$;} \\
    \{ Tree \}, & \hbox{otherwise.} \\
\end{array}%
\right.
\end{align*}
for all $m \in Tree$. By Definition \ref{sigma}.2, therefore, the
second conjunct in the antecedent of \eqref{reg} will be trivially
falsified. This means that $F \in \mathcal{F}^{Ag}_{reg}$, as
desired.
\end{proof}

Next we recall the definition of the Hilbert-style axiomatic
system $\Sigma_D$ from Part I. We first fix an arbitrary agent
community $Ag$ (and will keep it fixed till  Section
\ref{conclusion}). The set of axiom schemes for $\Sigma_D$ then
looks as follows:

\begin{align}
&\textup{A full set of axioms for classical propositional
logic}\label{A0}\tag{\text{A0}}\\
&\textup{$S5$ axioms for $\Box$ and $[j]$ for every $j \in
Ag$}\label{A1}\tag{\text{A1}}\\
&\Box A \to [j]A \textup{ for every }j \in Ag\label{A2}\tag{\text{A2}}\\
&(\Diamond[j_1]A_1 \wedge\ldots \wedge \Diamond[j_n]A_n) \to
\Diamond([j_1]A_1 \wedge\ldots \wedge[j_n]A_n)\label{A3}\tag{\text{A3}}\\
&(s\co(A \to B) \to (t\co A \to (s\times t)\co
B)\label{A4}\tag{\text{A4}}\\
&t\co A \to (!t\co(t\co A) \wedge KA)\label{A5}\tag{\text{A5}}\\
&(s\co A \vee t\co A) \to (s+t)\co A\label{A6}\tag{\text{A6}}\\
&\textup{$S4$ axioms for $K$}\label{A7}\tag{\text{A7}}\\
&KA \to \Box K\Box A\label{A8}\tag{\text{A8}}\\
&\Box Et \to K\Box Et\label{A9}\tag{\text{A9}}
\end{align}

The assumption is that in \eqref{A3} $j_1,\ldots, j_n$ are
pairwise different.

The rules of inferences are then as follows:

\begin{align}
&\textup{From }A, A \to B \textup{ infer } B;\label{R1}\tag{\text{R1}}\\
&\textup{From }A\textup{ infer }KA;\label{R2}\tag{\text{R2}}\\
&\textup{From }KA \to (\neg\Box Et_1 \vee\ldots\vee\neg\Box Et_n\vee\Box Es_1 \vee\ldots\vee\Box Es_k) \notag\\
&\qquad\qquad\textup{ infer }\Rightarrow KA \to (\neg Et_1
\vee\ldots\vee\neg Et_n\vee Es_1 \vee\ldots\vee
Es_k).\label{R_D}\tag{\text{$R_D$}}
\end{align}

$\Sigma_D$ is a minimal system in which we make no assumptions as
to the properties of proof constants. One standard way to extend
this minimal system (following a pattern established in the pure
justification logic) is to add a number of assumptions about proof
constants. More precisely, let us call a \emph{constant
specification} any set $\mathcal{CS}$ such that:
\begin{itemize}
\item $\mathcal{CS} \subseteq \{ c_n\co\ldots c_1\co A\mid
c_1,\ldots,c_n \in PConst, A\textup{ an instance of
}\eqref{A0}-\eqref{A9}\}$;

\item Whenever $c_{n+1}\co c_n\co\ldots c_1\co A \in
\mathcal{CS}$, then also $c_n\co\ldots c_1\co A \in \mathcal{CS}$.
\end{itemize}
For a given constant specification, we can define the
corresponding inference rule $R_{\mathcal{CS}}$ as follows:
\begin{align}
\textup{If }c_n\co\ldots c_1\co A \in \mathcal{CS},\textup{ infer
} c_n\co\ldots c_1\co A.\label{RCS}\tag{\text{$R_{\mathcal{CS}}$}}
\end{align}

We now define that $\Sigma_D(\mathcal{CS})$ is just $\Sigma_D$
extended with the rule \eqref{RCS}. Since $\emptyset$ is clearly
one example of constant specification, we have that
$\Sigma_D(\emptyset) = \Sigma_D$ so that our initial axiomatic
system is also in the class of systems of the form
$\Sigma_D(\mathcal{CS})$. However, when $\mathcal{CS} \neq
\emptyset$, the corresponding system $\Sigma_D(\mathcal{CS})$ will
prove some formulas which are not valid even if we restrict our
attention to jstit models based on any class of frames defined in
Section \ref{basic}. We therefore have to describe the restriction
on jstit models which comes with a commitment to a given
$\mathcal{CS}$. We say that a jstit model $\mathcal{M}$ is
$\mathcal{CS}$\emph{-normal} iff it is true that:
\begin{align*}
(\forall c \in PConst)(\forall m \in Tree)(\{ A \mid c\co A\in
\mathcal{CS} \} \subseteq \mathcal{E}(m,c)),
\end{align*}
where $\mathcal{E}$ is the $\mathcal{M}$'s admissible evidence
function. Again, it is easy to see that the class of
$\emptyset$-normal jstit models is just the whole class of jstit
models so that the representation $\Sigma_D(\emptyset) = \Sigma_D$
does not place any additional restrictions on the class of
intended models of $\Sigma_D$. Whenever $\mathcal{F}$ is a class
of frames, jstit or stit, we will denote the class of
$\mathcal{CS}$-normal jstit models based on the frames from
$\mathcal{F}$ by $Mod_{\mathcal{CS}}(\mathcal{F})$.

We then define a \emph{proof} in $\Sigma_D(\mathcal{CS})$ as a
finite sequence of formulas such that every formula in it is
either an axiom or is obtained from earlier elements of the
sequence by one of inference rules. A proof is a proof of its last
formula. If an $A \in Form^{Ag}$ is provable in
$\Sigma_D(\mathcal{CS})$, we will write $\vdash_{\mathcal{CS}} A$.
We say that $\Gamma \subseteq Form^{Ag}$ is \emph{inconsistent in
$\Sigma_D(\mathcal{CS})$} (or $\mathcal{CS}$-inconsistent) iff for
some $A_1,\ldots,A_n \in \Gamma$ we have $\vdash_{\mathcal{CS}}
(A_1 \wedge\ldots \wedge A_n) \to \bot$, and we say that $\Gamma$
is consistent in $\Sigma_D(\mathcal{CS})$ (or
$\mathcal{CS}$-consistent) iff it is not inconsistent in
$\Sigma_D(\mathcal{CS})$.

Finally, we cite (in a somewhat weakened form) the main result of
Part I, which we will use in this part without a proof:
\begin{theorem}\label{completeness}
Let $\Gamma \subseteq Form^{Ag}$. Then $\Gamma$ is
$\mathcal{CS}$-consistent iff it is satisfiable in
$Mod_{\mathcal{CS}}(\mathcal{C}^{Ag}_{mixsucc})$ iff it is
satisfiable in
$Mod^\downarrow_{\mathcal{CS}}(\mathcal{C}^{Ag}_{mixsucc})$.
\end{theorem}
We can immediately state a similar completeness result for the
temporal frames:
\begin{corollary}\label{c-completeness}
Let $\Gamma \subseteq Form^{Ag}$. Then $\Gamma$ is
$\mathcal{CS}$-consistent iff it is satisfiable in
$Mod_{\mathcal{CS}}(\mathcal{T}^{Ag}_{mixsucc})$ iff it is
satisfiable in
$Mod^\downarrow_{\mathcal{CS}}(\mathcal{T}^{Ag}_{mixsucc})$.
\end{corollary}
\begin{proof}
Note that $\mathcal{M}$ is a (unirelational) jstit model based on
a frame from $\mathcal{C}^{Ag}_{mixsucc}$ iff $\mathcal{M}$ is a
(unirelational) jstit model based on a frame from
$\mathcal{T}^{Ag}_{mixsucc}$.
\end{proof}

\section{Frame definability results}\label{main}

\subsection{Temporal and stit frames}
We deal with the stit frames first. The principal lemma looks as
follows:
\begin{lemma}\label{stit-falsify}
Let $\mathcal{CS}$ be a constant specification and let $C  =
\langle Tree, \unlhd, Choice\rangle$ be a stit frame outside
$\mathcal{C}^{Ag}_{mixsucc}$. Then there is a
$\mathcal{CS}$-normal jstit model $\mathcal{M}  = \langle Tree,
\unlhd, Choice, Act, R, R_e, \mathcal{E}, V\rangle$ based on $C$
such that for some $(m,h) \in MH(\mathcal{M})$ it is true that:
$$
\mathcal{M}, m, h \not\models K(\Box Ex \vee \neg\Box Ey) \to (Ex
\vee \neg Ey).
$$
\end{lemma}
\begin{proof}
Assume that $C \notin \mathcal{C}^{Ag}_{mixsucc}$. Then we can
choose $m_0,m_1 \in Tree$ and $h_0, h_1 \in H_{m_0}$ such that:

\begin{equation}\label{E:u1}
(h_0 \not\approx_{m_0} h_1) \& (m_0 \lhd m_1) \& (\forall m \unlhd
m_1)(\neg Next(m_0, m)).
\end{equation}

We now extend $C$ to $\mathcal{M}$ setting $R = R_e = \unlhd$,
$\mathcal{E}(m,t) = Form^{Ag}$ for all $m \in Tree$ and $t \in
Pol$, and setting $V(p) = \emptyset$ for all $p \in Var$. As for
$Act$, we set as follows. We first choose an arbitrary $h_2 \in
H_{m_1}$. By Lemma \ref{technical}.2 we know that also $h_2 \in
H_{m_0}$. Now for an arbitrary $m \in Tree$ we define that:
\begin{align*}
    Act(m,h) = \left\{%
\begin{array}{ll}
    \{ y \}, & \hbox{if $m = m_0$ and $h \approx_{m_0} h_2$;} \\
    \{ x,y \}, & \hbox{if $m \rhd m_0$ and $h \approx_{m_0} h_2$;} \\
    \emptyset, & \hbox{otherwise.} \\
\end{array}%
\right.
\end{align*}
It is obvious that every semantical constraint on jstit models is
satisfied, except possibly for the constraints invoking $Act$, and
it is also clear that such an $\mathcal{M}$ satisfies
$\mathcal{CS}$-normality condition for every constant possible
specification $\mathcal{CS}$.

As for $Act$ itself, we start by establishing the following:

\emph{Claim}. Under the current settings for $\mathcal{M}$ we
have, for an arbitrary $m \in Tree$:
\begin{align*}
    Act_m = \left\{%
\begin{array}{ll}
    \{ x,y \}, & (\exists h \in H_m)(m \rhd m_0 \& h \approx_{m_0} h_2); \\
    \emptyset, & \hbox{otherwise.} \\
\end{array}%
\right.
\end{align*}
Indeed, assume that $m \in Tree$ and $h \in H_m$ are such that $m
\rhd m_0$ and $h \approx_{m_0} h_2$. Now, if $g \in H_m$ is
arbitrary, then, by Lemma \ref{technical}.2, $g,h \in H_{m_0}$ so
that $g \approx_{m_0} h$. By Lemma \ref{technical}.3, we get then
$g \approx_{m_0} h_2$ so that $Act(m,g) = \{ x,y \}$. Since $g \in
H_m$ was chosen arbitrarily, this means that also $Act_m = \{ x,y
\}$.

On the other hand, if either $m_0 \ntrianglelefteq m$ or no
history in $H_m$ is undivided from $h_2$ at $m_0$, then we
obviously have $Act_m = \emptyset$. Assume then that $m = m_0$.
Recall that $h_0, h_1 \in H_{m_0} = H_m$ are such that $h_0
\not\approx_{m_0} h_1$. Therefore, by Lemma \ref{technical}.3, we
must have either $h_0 \not\approx_{m_0} h_2$ or $h_1
\not\approx_{m_0} h_2$, whence either $Act(m_0,h_0)$ or
$Act(m_0,h_1)$ equals to $\emptyset$. In any case, we will have
$\emptyset = Act_{m_0} = Act_m$.

We now look into the semantical constraints dependent on $Act$ in
some detail.

\textbf{Expansion of presented proofs}. Assume that $m \lhd m'$
and that $h \in H_{m'}$. Then also $h \in H_{m}$ by Lemma
\ref{technical}.2. Now, if $m_0 \ntrianglelefteq m$, then $Act(m,
h) = \emptyset$ and the constraint is verified trivially. The same
argument applies, if $h \not\approx_{m_0} h_2$. Further, if $m =
m_0$ and $h\approx_{m_0} h_2$, then we must have $Act(m, h) = \{y
\}$ and $Act(m', h) = \{x, y\}$, respectively, and the constraint
is satisfied. Finally, if $m \rhd m_0$ and $h\approx_{m_0} h_2$,
then we must have $Act(m, h) = Act(m', h) = \{x, y\}$, and the
constraint is again satisfied.

\textbf{Presenting a new proof makes histories divide}. Assume
that $m \in Tree$ and that $h\approx_{m} g$, so that for some $m'
\rhd m$ it is true that $m' \in h \cap g$. Now, if $m_0
\ntrianglelefteq m$, then $Act(m,h) = Act(m, g) = \emptyset$ and
the constraint is verified. The same argument applies when $h,g
\not\approx_{m_0} h_2$. Finally, if $m_0 \unlhd m$ and $h,g
\approx_{m_0} h_2$, then either $Act(m,h) = Act(m, g) = \{ y \}$
or $Act(m,h) = Act(m, g) = \{ x,y \}$ depending on whether $m =
m_0$ or $m \rhd m_0$.

\textbf{No new proofs guaranteed}. Assume that $m \in Tree$. If
$Act_m = \emptyset$, then the constraint is trivially satisfied.
On the other hand, if $Act_m \neq \emptyset$, then, by the Claim
above, we must have $Act_m = \{ x,y \}$ and also that $m \rhd m_0
\& h \approx_{m_0} h_2$ for some $h \in H_m$. Therefore, we can
choose an $m' \rhd m_0$ such that $m' \in h_2 \cap h$. But then
$m'$ must be $\unlhd$-comparable with $m$ and we need to deal with
the two cases:

\emph{Case 1}. $m \unlhd m'$. Then, by Lemma \ref{technical}.2,
$h_2 \in H_m$. Recall that, by its choice, $h_2 \in H_{m_1}$, so
that $m$ must be $\unlhd$-comparable with $m_1$ as well. By the
Claim above, we clearly have $Act(m_1, h_2) = \{ x,y \}$,
therefore, if $m_1 \lhd m$, then we are done. On the other hand,
if $m \unlhd m_1$, then, by \eqref{E:u1}, $\neg Next(m_0, m)$. The
latter means that we can choose an $m'' \lhd m$ such that $m''
\ntrianglelefteq m_0$. Hence by the absence of backward branching,
we will have $m_0 \lhd m''$. Note that by Lemma \ref{technical}.2
and  $m'' \lhd m \unlhd m_1$ we will also have $h_2 \in H_{m''}$
whence, by $m_0 \lhd m''$, we get that $Act(m'', h_2) = \{ x,y \}$
again satisfying the constraint.

\emph{Case 2}. $m' \lhd m$. Then, since $m' \rhd m_0$ and $h
\approx_{m_0} h_2$, we must have $Act(m',h) = \{ x,y \}$ thus
satisfying the constraint.

\textbf{Presented proofs are epistemically transparent}. Assume
that $m,m' \in Tree$ are such that $R_e(m,m')$. Then, by
definition of $R_e$ above, we will also have $m \unlhd m'$. Now,
if $Act_m$ is empty, then the constraint is trivially verified.
Otherwise we will have $Act_m = \{ x,y \}$ by the Claim above. Let
$h \in H_{m'}$ be arbitrary. By Lemma \ref{technical}.2, $h \in
H_m$, therefore $Act(m,h)  = \{ x,y \}$. But then, by the
expansion of presented proofs constraint (verified above) we must
have $Act(m',h) = \{ x,y \}$. Since $h \in H_{m'}$ was chosen
arbitrarily, this shows that $Act_{m'} = \{ x,y \}$ and the
constraint is satisfied.

Therefore, the above-defined $\mathcal{M}$ is shown to be a jstit
model for $Ag$ and by the Claim above we obviously have that:
$$
\mathcal{M}, m_0, h_2 \not\models K(\Box Ex \vee \neg\Box Ey) \to
(Ex \vee \neg Ey).
$$
Indeed, whenever $m' \in Tree$, then, by the Claim above, we will
either have $Act_{m'} = \emptyset$ (and then $\mathcal{M}, m', g
\models \neg\Box Ey$ for all $g \in H_{m'}$), or  $Act_{m'} = \{
x,y \}$ (and then $\mathcal{M}, m', g \models \Box Ex$ for all $g
\in H_{m'}$). Therefore, it is clear that we have:
$$
\mathcal{M}, m_0, h_2 \models K(\Box Ex \vee \neg\Box Ey),
$$
and yet, on the other hand it is true that:
$$
\mathcal{M}, m_0, h_2 \models \neg Ex \wedge Ey.
$$
\end{proof}
The frame definability result for stit frames is now
straightforward:
\begin{theorem}\label{frame1}
Let $C  = \langle Tree, \unlhd, Choice\rangle$ be a stit frame for
$Ag$. For any constant specification $\mathcal{CS}$ it is true
that:
$$
(\forall \mathcal{M} \in Mod_{\mathcal{CS}}(\{ C \}))(\mathcal{M}
\models \{ A \in Form^{Ag} \mid \vdash_{\mathcal{CS}} A \})
\Leftrightarrow C \in \mathcal{C}^{Ag}_{mixsucc}.
$$
\end{theorem}
\begin{proof}
The $(\Leftarrow)$-part follows from Theorem \ref{completeness}.
For the $(\Rightarrow)$-part, note that for a given $x,y \in PVar$
we have
$$
\vdash_{\mathcal{CS}} K(\Box Ex \vee \neg\Box Ey) \to (Ex
\vee \neg Ey)
$$
by one application of \eqref{R_D} to an appropriate instance of
\eqref{A7}. Given this fact, the $(\Rightarrow)$-part follows from
Lemma \ref{stit-falsify}.
\end{proof}

An analogous result for temporal frames is an easy corollary of
the facts established above. More precisely, we claim the
following:

\begin{corollary}\label{temp-falsify}
Let $\mathcal{CS}$ be a constant specification and let $T  =
\langle Tree, \unlhd\rangle$ be a temporal frame outside
$\mathcal{T}^{Ag}_{mixsucc}$. Then there is a
$\mathcal{CS}$-normal jstit model $\mathcal{M}  = \langle Tree,
\unlhd, Choice, Act, R, R_e, \mathcal{E}, V\rangle$ based on $C$
such that for some $(m,h) \in MH(\mathcal{M})$ it is true that:
$$
\mathcal{M}, m, h \not\models K(\Box Ex \vee \neg\Box Ey) \to (Ex
\vee \neg Ey).
$$
\end{corollary}
\begin{proof}
Just repeat the proof of Lemma \ref{stit-falsify} adding to the
definition of $\mathcal{M}$ that we set $Choice^m_j = H_m$ for all
$m \in Tree$ and $j \in Ag$.
\end{proof}
Now we can establish the following theorem in the same way as
Theorem \ref{frame1}, using Corollaries \ref{c-completeness} and
\ref{temp-falsify} instead of Theorem \ref{completeness} and Lemma
\ref{stit-falsify}, respectively:
\begin{theorem}\label{frame1-t}
Let $T  = \langle Tree, \unlhd\rangle$ be a temporal frame for
$Ag$. For any constant specification $\mathcal{CS}$ it is true
that:
$$
(\forall \mathcal{M} \in Mod_{\mathcal{CS}}(\{ T \})(\mathcal{M}
\models \{ A \in Form^{Ag} \mid \vdash_{\mathcal{CS}} A \})
\Leftrightarrow T \in \mathcal{T}^{Ag}_{mixsucc}.
$$
\end{theorem}

\subsection{Justification stit frames}

We now turn to the much more complex case of jstit frames. First,
we need to know how $\Sigma_D(\mathcal{CS})$ stands in relation to
the $\mathcal{CS}$-normal models based on regular jstit frames,
and we start answering this question by establishing a soundness
claim. This claim mostly reduces to a routine check that every
axiom is valid and that rules preserve validity. We treat the less
obvious cases in some detail:

\begin{theorem}\label{soundness}
Let $\mathcal{CS}$ be an arbitrary constant specification. Then
every instance of \eqref{A0}--\eqref{A9} is valid over the class
$Mod_{\mathcal{CS}}(\mathcal{F}^{Ag}_{reg})$, and every
application of rules \eqref{R1},\eqref{R2},\eqref{R_D}, and
\eqref{RCS} to formulas which are valid over
$Mod_{\mathcal{CS}}(\mathcal{F}^{Ag}_{reg})$ yields a formula
which is valid over the same class.
\end{theorem}
\begin{proof}
First, note that if $\mathcal{M} = \langle Tree, \unlhd, Choice,
Act, R, R_e, \mathcal{E}, V\rangle$ is a $\mathcal{CS}$-normal
jstit model based on a jstit frame from $\mathcal{F}^{Ag}_{reg}$,
then $\langle Tree, \unlhd, Choice, V\rangle$ is a model of stit
logic. Therefore, axioms \eqref{A0}--\eqref{A3}, which were
copy-pasted from the standard axiomatization of \emph{dstit}
logic\footnote{See, e.g. \cite[Ch. 17]{belnap2001facing}, although
$\Sigma$ uses a simpler format closer to that given in
\cite[Section 2.3]{balbiani}.} must be valid. Second, note that if
$\mathcal{M} = \langle Tree, \unlhd, Choice, Act, R, R_e,
\mathcal{E}, V\rangle$ is a $\mathcal{CS}$-normal jstit model,
then $\mathcal{M} = \langle Tree, R, R_e, \mathcal{E}, V\rangle$
is what is called in \cite[Section 6]{ArtemovN05} a justification
model with the form of constant specification given by
$\mathcal{CS}$\footnote{The format for the variable assignment $V$
is slightly different, but this is of no consequence for the
present setting.}. This means that also all of the
\eqref{A4}--\eqref{A7} must be valid, whereas \eqref{R1},
\eqref{R2}, and \eqref{RCS} must preserve validity, given that all
these parts of our axiomatic system were borrowed from the
standard axiomatization of justification logic. The validity of
other parts of $\Sigma_D(\mathcal{CS})$ will be motivated below in
some detail. In what follows, $\mathcal{M} = \langle Tree, \unlhd,
Choice, Act, R, R_e,\mathcal{E}, V\rangle$ will always stand for
an arbitrary jstit model in
$Mod_{\mathcal{CS}}(\mathcal{F}^{Ag}_{reg})$, and $(m,h)$ for an
arbitrary element of $MH(\mathcal{M})$.

As for \eqref{A8}, assume for \emph{reductio} that $\mathcal{M},
m,h \models KA \wedge \Diamond K \Diamond\neg A$. Then
$\mathcal{M}, m,h \models KA$ and also $\mathcal{M}, m,h' \models
K \Diamond\neg A$ for some $h' \in H_m$. By reflexivity of $R$, it
follows that $\Diamond\neg A$ will be satisfied at $(m,h)$ in
$\mathcal{M}$. The latter means that, for some $h'' \in H_m$, $A$
must fail at $(m,h'')$ and therefore, again by reflexivity of $R$,
$KA$ must fail at $(m,h)$ in $\mathcal{M}$, a contradiction.

We consider next \eqref{A9}. If $\Box Et$ is true at $(m,h)$ in
$\mathcal{M}$, then, by definition, $t \in Act_m$. Now, if $m' \in
Tree$ is such that $R(m,m')$, then, by $R \subseteq R_e$ we will
have $R_e(m,m')$, and, by the epistemic transparency of presented
proofs constraint, we must have $t \in Act_{m'}$ so that for every
$g \in H_{m'}$ we will have $\mathcal{M},m',g \models \Box Et$.
Therefore, we must have $\mathcal{M},m,h \models K\Box Et$ as
well.

The hardest part is to show that \eqref{R_D} preserves validity
over jstit models from
$Mod_{\mathcal{CS}}(\mathcal{F}^{Ag}_{reg})$. Assume that $KA \to
(\neg\Box Et_1 \vee\ldots\vee\neg\Box Et_n\vee\Box Es_1
\vee\ldots\vee\Box Es_k)$ is valid over this class of jstit
models, and assume also that we have:

\begin{equation}\label{E:e1}
\mathcal{M}, m, h \models KA \wedge Et_1 \wedge\ldots\wedge
Et_n\wedge\neg Es_1 \wedge\ldots\wedge\neg Es_k.
\end{equation}

By validity of \eqref{A1}, it follows that:
$$
\mathcal{M}, m, h \models KA \wedge \neg\Box Es_1
\wedge\ldots\wedge\neg\Box Es_k.
$$
Whence, by the assumed validity we know that also:
$$
\mathcal{M}, m, h \models \neg\Box Et_1 \vee\ldots\vee\neg\Box
Et_n,
$$
therefore, we can choose a natural $u$ such that $1 \leq u \leq n$
and:
$$
\mathcal{M}, m, h \models \neg\Box Et_u.
$$
The latter, in turn, means that for some $h' \in H_m$ we have
that:

\begin{equation}\label{E:e2}
\mathcal{M}, m, h' \models \neg Et_u.
\end{equation}
Comparison between \eqref{E:e1} and \eqref{E:e2} shows that
$Act(m,h) \neq Act(m,h')$, whence by the presenting a new proof
makes histories divide constraint we get that $h \not\approx_m
h'$. Hence we know that $m$ cannot be $\unlhd$-maximal in $Tree$.
Using Lemma \ref{technical}.1, we can choose in $Tree$ some $m_1
\rhd m$ such that $m_1 \in h$. We now establish the following
claims:

\emph{Claim 1}. $(\forall g \in H_{m_1})(h' \not\approx_m g)$.

The argument is the same as for $h$: if $g \in H_{m_1}$, then $m_1
\in g \cap h$ so that, by $m_1 \rhd m$ we must have $g \approx_m
h$. But then, given Lemma \ref{technical}.3 and $h \not\approx_m
h'$, we cannot have $g \approx_m h'$.

\emph{Claim 2}. $S = \{ m'' \in Tree\mid t_1,\ldots, t_n \in
Act_{m''} \} \in \Theta_{m'}$ for every $m'$ such that $m \lhd m'
\unlhd m_1$. Furthermore, $m \notin S$.

The fact that $m \notin S$ immediately follows from \eqref{E:e2}.
Now, choose in $Tree$ an arbitrary $m'$ such that $m \lhd m'
\unlhd m_1$. Since $h \in H_{m_1}$, we know, by Lemma
\ref{technical}.2, that $h \in H_{m'}$. Further, if $g \in H_{m'}$
is arbitrary, then, by the same lemma, $g \in H_m$. Therefore, $g
\approx_m h$ and, by the presenting a new proof makes histories
divide constraint, $Act(m,g) = Act(m,h) \supseteq \{ t_1,\ldots,
t_n \}$. Whence, by $m \lhd m'$ and the expansion of presented
proofs constraints we get that $\{ t_1,\ldots, t_n \} \subseteq
Act(m',g)$. Since $g \in H_{m'}$ was chosen arbitrarily, this
means that $t_1,\ldots, t_n \in Act_{m'}$ and hence $m' \in S$
thus verifying Definition \ref{sigma}.1.

Next, if $m_2 \in S$ and $R_e(m_2,m_3)$, then $t_1,\ldots, t_n \in
Act_{m_2}$, hence by the epistemic transparency of presented
proofs $t_1,\ldots, t_n \in Act_{m_3}$, which means that also $m_3
\in S$ and Definition \ref{sigma}.2 is also verified.

Furthermore, assume that $m_2 \in Tree$ is such that, for all $g
\in H_{m_2}$, there exists $m_g \in g$ with the property
$Next(m_2, m_g) \& m_g \in S$. So choose an arbitrary $g \in
H_{m_2}$. We have then $t_1,\ldots, t_n \in Act_{m_g}$, whence, by
the no new proofs guaranteed constraint we can choose $(m^1_g,
h^1_g),\ldots, (m^n_g, h^n_g)$ such that:
\begin{equation}\label{E:eq-a}
\begin{array}{r@{}l}
h^1_g \in H_{m_g} \& m^1_g \lhd m_g \& t_1 &\in Act(m^1_g, h^1_g);\\
&\ldots\\
h^n_g \in H_{m_g} \& m^n_g \lhd m_g \& t_n &\in Act(m^n_g, h^n_g).
\end{array}
\end{equation}
By $Next(m_2, m_g)$ and \eqref{E:eq-a} we get that:
\begin{equation}\label{E:e3}
    m^1_g, \ldots, m^n_g \unlhd m_2.
\end{equation}
From $g, h^1_g, \ldots, h^n_g \in H_{m_g}$, $m_2 \lhd m_g$, and
Lemma \ref{technical}.2 we get that:
\begin{equation}\label{E:e4}
    h^1_g, \ldots, h^n_g \in H_{m_2}.
\end{equation}
and, further:
\begin{equation}\label{E:e5}
    h^1_g \approx_{m_2} g, \ldots, h^n_g \approx_{m_2} g.
\end{equation}
By the presenting a new proof makes histories divide constraint,
this further means that:
\begin{equation}\label{E:e6}
    Act(m_2, h^1_g) = \ldots = Act(m_2, h^n_g) = Act(m_2, g).
\end{equation}
Next, by \eqref{E:e3}, \eqref{E:e4}, the expansion of presented
proofs constraint, and \eqref{E:eq-a} we get that:
\begin{equation}\label{E:e7}
   t_1 \in Act(m_2, h^1_g), \ldots, t_n \in Act(m_2, h^n_g).
\end{equation}
It follows now from \eqref{E:e6} and \eqref{E:e7} that $t_1,
\ldots, t_n \in Act(m_2, g)$. Since $g \in H_{m_2}$ was chosen
arbitrarily, this further means that $t_1, \ldots, t_n \in
Act_{m_2}$ and thus $m_2 \in S$, as desired. In this way,
Definition \ref{sigma}.3 is verified.

Now, let $m_2 \in S$ and assume that:
\begin{equation}\label{E:e9}
(\forall m_3 \lhd m_2)\exists m_4(m_3 \lhd m_4 \lhd m_2).
\end{equation}
By $m_2 \in S$ we know that $t_1,\ldots, t_n \in Act_{m_2}$.
Again, by the no new proofs guaranteed constraint we can choose
$(m^1, h^1),\ldots, (m^n, h^n)$ such that:
\begin{equation}\label{E:eq-a1}
\begin{array}{r@{}l}
h^1 \in H_{m_2} \& m^1 \lhd m_2 \& t_1 &\in Act(m^1, h^1);\\
&\ldots\\
h^n \in H_{m_2} \& m^n \lhd m_2 \& t_n &\in Act(m^n, h^n).
\end{array}
\end{equation}
By \eqref{E:eq-a1} and the absence of backward branching it
follows that all of $m^1,\ldots, m^n$ are $\unlhd$-comparable, so
we let $m'$ be the $\unlhd$-greatest moment among $m^1,\ldots,
m^n$.  By the choice of $m'$ and \eqref{E:eq-a1}, we have:
\begin{equation}\label{E:e10}
   m' \lhd m_2.
\end{equation}
Therefore, by Lemma \ref{technical}.2, we get $H_{m_2} \subseteq
H_{m'}$, whence:
\begin{equation}\label{E:e11}
    h^1, \ldots, h^n \in H_{m'}.
\end{equation}
It follows then from \eqref{E:eq-a1} and \eqref{E:e10} that:
\begin{equation}\label{E:e12}
    h^1 \approx_{m'} \ldots \approx_{m'} h^n,
\end{equation}
which further means, by the presenting a new proof makes histories
divide constraint that:
\begin{equation}\label{E:e13}
    Act(m', h^1) = \ldots = Act(m', h^n).
\end{equation}
Again by the choice of $m'$ and the expansion of presented proofs
constraint, we further get that:
\begin{equation}\label{E:e9-1}
   t_1 \in Act(m', h^1), \ldots, t_n \in Act(m', h^n).
\end{equation}

It follows then from \eqref{E:e9-1} and \eqref{E:e13} that
$t_1,\ldots, t_n \in Act(m', h^1)$. Now, by \eqref{E:e9} and
\eqref{E:e10}, we can choose an $m'' \in Tree$ such that $m' \lhd
m'' \lhd m_2$. By Lemma \ref{technical}.2, we know that $H_{m_2}
\subseteq H_{m''}$, whence, by \eqref{E:eq-a1}, $h^1 \in H_{m''}$.
It follows, by Lemma \ref{technical-new}, that $t_1,\ldots, t_n
\in Act_{m''}$ and thus $m'' \in S$, as desired. This ends both
the verification of Definition \ref{sigma}.4 and the proof of
Claim 2.

\emph{Claim 3}. $(\forall m'_1 \in h')(Next(m,m'_1) \Rightarrow
m'_1 \notin S)$.

Indeed, assume the contrary, i.e. that for some  $m'_1 \in h'$ we
have both $Next(m,m'_1)$ and $m'_1 \in S$. Then  we will have
$t_1,\ldots, t_n \in Act_{m'_1}$. But then, by the no new proofs
guaranteed constraint, we can choose a $g \in H_{m'_1}$ and $m''
\lhd m'_1$ such that $t_u \in Act(m'',g)$. By $Next(m,m'_1)$ we
know that $m'' \unlhd m$ and by Lemma \ref{technical}.2 and $m'_1
\rhd m$ we know that $g \in H_m$. Therefore, by the expansion of
presented proofs, we get that $t_u \in Act(m,g)$. Moreover, note
that $m'_1 \in h' \cap g$ so that $h' \approx_m g$. Therefore, by
the presenting a new proof makes histories divide constraint, we
must have $Act(m, h') = Act(m,g) \ni t_u$ which is in plain
contradiction with \eqref{E:e2}.

In view of the Claims 1--3 above, we must be able to choose an
$m_2 \in Tree$ such that both $m_2 \unlhd m_1$ and $Next(m, m_2)$.
So we consider such an $m_2$. Given that $m_1 \in h$, we know, by
Lemma \ref{technical}.2, that $h \in H_{m_2}$. Therefore, it
follows from \eqref{E:e1} and Lemma \ref{technical-new} that
$t_1,\ldots, t_n \in Act_{m_2}$, or, equivalently:
\begin{equation}\label{E:e14}
   \mathcal{M}, m_2, h \models \Box Et_1 \wedge\ldots \wedge \Box
   Et_n.
\end{equation}
Furthermore, by the future always matters constraint we know that
$R(m,m_2)$, whence it follows, again by \eqref{E:e1}, that:
\begin{equation}\label{E:e15}
   \mathcal{M}, m_2, h \models KA.
\end{equation}
Finally, choose an arbitrary $r$ between $1$ and $k$. If $s_r \in
Act_{m_2}$, then, by the no new proofs guaranteed constraint,
there must be some $g \in H_{m_2}$ and some $m_0 \lhd m$ such that
$s_r \in Act(m_0, g)$. Then, by Lemma \ref{technical}.2, $g \in
H_m$, hence $h \approx_m g$. Therefore, by the presenting a new
proof makes histories divide constraint, $Act(m, g) = Act(m,h)$.
By $Next(m, m_2)$ we must have $m_0 \unlhd m$, therefore, by the
expansion of presented proofs, $s_r \in Act(m,g)$, whence also
$s_r \in Act(m,h)$. But this plainly contradicts \eqref{E:e1}.
Since $1 \leq r \leq k$ was chosen arbitrarily, this means that
all of $s_1,\ldots, s_k$ are outside $Act_{m_2}$ so that we have:
\begin{equation}\label{E:e16}
   \mathcal{M}, m_2, h \models \neg\Box Es_1 \wedge\ldots \wedge \neg\Box
   Es_k.
\end{equation}
Taken together, \eqref{E:e14}--\eqref{E:e16} contradict the
assumed validity of

\noindent$KA \to (\neg\Box Et_1 \vee\ldots\vee\neg\Box
Et_n\vee\Box Es_1 \vee\ldots\vee\Box Es_k)$.
\end{proof}

It follows from Theorem \ref{soundness} that we cannot have a
result analogous to Theorem \ref{frame1} w.r.t. jstit frames.
Indeed, by Lemma \ref{stit-jstit}.2, there exists a regular jstit
frame $F$ for $Ag$, which violates condition \eqref{mixsucc}.
However, by Theorem \ref{soundness}, every $\mathcal{CS}$-normal
jstit model based on $F$ will make every theorem of
$\Sigma_D(\mathcal{CS})$ valid. Therefore, the frame definability
result for jstit frames has to use a much more involved regularity
condition in place of \eqref{mixsucc}.

Even though Theorem \ref{soundness} is already sufficient to
derive the frame definability theorem for jstit frames, we pause
to observe that one can actually get a completeness theorem as
well:
\begin{theorem}\label{j-completeness}
Let $\Gamma \subseteq Form^{Ag}$ and let $\mathcal{F}$ be a class
of jstit frames such that $\mathcal{F}^{Ag}_{reg}\downarrow
\subseteq \mathcal{F} \subseteq \mathcal{F}^{Ag}_{reg}$. Then
$\Gamma$ is $\mathcal{CS}$-consistent iff it is satisfiable in
$Mod_{\mathcal{CS}}(\mathcal{F})$.
\end{theorem}
\begin{proof}
($\Rightarrow$). Let $\Gamma \subseteq Form^{Ag}$ be satisfiable
in $Mod_{\mathcal{CS}}(\mathcal{F})$ so that for some $\mathcal{M}
\in Mod_{\mathcal{CS}}(\mathcal{F})$ and some $(m,h) \in
MH(\mathcal{M})$ we have $\mathcal{M}, m, h \models \Gamma$. Then
we must have $\mathcal{M} \in
Mod_{\mathcal{CS}}(\mathcal{F}^{Ag}_{reg})$. If $\Gamma$ were
$\mathcal{CS}$-inconsistent, this would mean that for some
$A_1,\ldots,A_n \in \Gamma$ we would have $\vdash_{\mathcal{CS}}
(A_1 \wedge\ldots \wedge A_n) \to \bot$. By Theorem
\ref{soundness}, this would mean that:
$$
\mathcal{M}, m, h \models (A_1 \wedge\ldots \wedge A_n) \to \bot,
$$
whence clearly $\mathcal{M}, m, h \models \bot$, which is
impossible. Therefore, $\Gamma$ must be $\mathcal{CS}$-consistent.

($\Leftarrow$). We can re-use the canonical model
$\mathcal{M}^{Ag}_{\mathcal{CS}}$ from Part I of this paper. In
Part I, $\mathcal{M}^{Ag}_{\mathcal{CS}}$ was shown to be
$\mathcal{CS}$-universal in that it satisfies every
$\mathcal{CS}$-consistent subset of $Form^{Ag}$. It was also shown
that $\mathcal{M}^{Ag}_{\mathcal{CS}} \in
Mod^\downarrow_{\mathcal{CS}}(\mathcal{C}^{Ag}_{mixsucc})$,
whence, by Lemma \ref{stit-jstit}.1, we get that
$\mathcal{M}^{Ag}_{\mathcal{CS}}$ is in
$Mod_{\mathcal{CS}}(\mathcal{F}^{Ag}_{reg}\downarrow)$ and
therefore in $\mathcal{F}$.
\end{proof}

Now for the frame definability for jstit frames:
\begin{lemma}\label{jstit-falsify}
$\mathcal{CS}$ be a constant specification and let $F  = \langle
Tree, \unlhd, Choice, R, R_e \rangle$ be a jstit frame outside
$\mathcal{F}^{Ag}_{reg}$. Then there is a $\mathcal{CS}$-normal
jstit model $\mathcal{M}  = \langle Tree, \unlhd, Choice, Act, R,
R_e, \mathcal{E}, V\rangle$ based on $F$ such that for some $(m,h)
\in MH(\mathcal{M})$ it is true that:
$$
\mathcal{M}, m, h \not\models K(\Box Ex \vee \neg\Box Ey) \to (Ex
\vee \neg Ey).
$$
\end{lemma}
\begin{proof}
Assume that $F \notin \mathcal{F}^{Ag}_{reg}$. Then we can choose
$m_0,m_1 \in Tree$, $h' \in H_{m_0}$, and $S \in\bigcap_{m_0 \lhd
m \unlhd m_1}\Theta_m$ such that:
\begin{align}
&(m_0 \lhd m_1) \& m_0 \notin S \& (\forall g \in H_{m_1})(g
\not\approx_{m_0} h') \&\notag\\
&\qquad\qquad\& (\forall m' \in h')(Next(m_0,m') \Rightarrow m'
\notin S)\& (\forall m \unlhd m_1)(\neg Next(m_0, m)).\label{E:u2}
\end{align}
We now extend $F$ to $\mathcal{M}$ setting $\mathcal{E}(m,t) =
Form^{Ag}$ for all $m \in Tree$ and $t \in Pol$, and setting $V(p)
= \emptyset$ for all $p \in Var$. As for $Act$, we set as follows.
We first choose an arbitrary $h_2 \in H_{m_1}$. By Lemma
\ref{technical}.2 we know that also $h_2 \in H_{m_0}$. Now for an
arbitrary $m \in Tree$ we define that:
\begin{align*}
    Act(m,h) = \left\{%
\begin{array}{ll}
    \{ y \}, & \hbox{if $m = m_0$ and $h \approx_{m_0} h_2$;} \\
    \{ x,y \}, &\hbox{if }m \in S \vee \exists m'(m' \in h \cap S \& Next(m,m')); \\
    \emptyset, & \hbox{otherwise.} \\
\end{array}%
\right.
\end{align*}
It is obvious that every semantical constraint on jstit models is
satisfied, except possibly for the constraints invoking $Act$, and
it is also clear that such an $\mathcal{M}$ satisfies
$\mathcal{CS}$-normality condition for every possible constant
specification $\mathcal{CS}$.

As for $Act$ itself, we start by establishing the following
claims:

\emph{Claim 1}. $(\forall m \in Tree)(m \in S \Leftrightarrow
Act_m = \{ x,y \})$.

Indeed, whenever $m \in S$, we will have $Act(m,h) = \{ x,y \}$
for every $h \in H_m$ and hence $Act_m = \{ x,y \}$. In the other
direction, assume that for every $h \in H_m$ it is true that
$Act(m,h) = \{ x,y \}$. If $m \in S$, then we are done. If $m
\notin S$, then for every $h \in H_m$ we must have an $m_h \in h$
such that both $m_h \in S$ and $Next(m,m_h)$. But then, by
Definition \ref{sigma}.3, we must also have $m \in S$ despite our
initial assumption.

\emph{Claim 2}. $(\forall m \in Tree)(Act_m = \emptyset \vee Act_m
= \{ x,y \})$.

Indeed, if $m \neq m_0$ then for every $h \in H_m$ we will have
either $Act(m,h) = \emptyset$ or $Act(m,h) = \{ x,y \}$ just by
definition of $Act$ so that the claim is obviously true. And if $m
= m_0$. then we know that $Act(m_0, h') = \emptyset$ so that we
must have $Act_{m_0} = \emptyset$.

\emph{Claim 3}. Under the settings for $\mathcal{M}$ we have, for
an arbitrary $m \in Tree$:
\begin{align*}
    Act_m = \left\{%
\begin{array}{ll}
    \{ x,y \}, & m \in S; \\
    \emptyset, & \hbox{otherwise.} \\
\end{array}%
\right.
\end{align*}
Immediate from Claims 1 and 2.

\emph{Claim 4}. $(\forall m \in Tree)((m \in h_2 \& m \rhd m_0)
\Rightarrow m \in S)$.

 Indeed, if $m \in h_2$, then $m$ must be $\unlhd$-comparable to
 $m_1$. Now, if $m \unlhd m_1$, then $m_0 \lhd m \unlhd m_1$ so
 that $m \in S$ by Definition \ref{sigma}.1. On the other
 hand, if $m_1 \lhd m$, then note that we clearly have $m_1 \in S$
 by Definition \ref{sigma}.1. By the future always matters
 constraint and $R \subseteq R_e$ we further get $R_e(m_1, m)$, whence by Definition
 \ref{sigma}.2
 we again get $m \in S$.

We now look into the semantical constraints dependent on $Act$ in
some detail.

\textbf{Expansion of presented proofs}. Assume that $m \lhd m'$
and $h \in H_{m'}$. We have three cases to consider.

\emph{Case 1}. $Act(m, h) = \emptyset$. The constraint is verified
trivially.

\emph{Case 2}. $Act(m,h) = \{ y \}$. Then $m = m_0$ and $h
\approx_{m_0} h_2$. The latter means that we can choose an $m''
\rhd m = m_0$ such that $m'' \in h \cap h_2$. By Claim 4, we get
then that $m'' \in S$. Now, since also $m' \in h$, $m'$ and $m''$
must be $\unlhd$-comparable. If $m'' \lhd m'$, then $R_e(m'', m')$
by the future always matters constraint and $R \subseteq R_e$,
and, further, $m' \in S$ by Definition \ref{sigma}.2. If $m'
\unlhd m''$ then by Lemma \ref{technical}.2, we get that $m' \in
h_2$ and since also $m' \rhd m = m_0$, this means that  $m' \in S$
by Claim 4. Thus we get $m' \in S$ anyway, which means that
$Act(m',h) = \{ x,y \}$ and the constraint is satisfied.

\emph{Case 3}. $Act(m,h) = \{ x,y \}$. If $m \in S$ then also $m'
\in S$ by $R \subseteq R_e$, the future always matters constraint,
and Definition \ref{sigma}.2. On the other hand, if there exists
$m'' \in h$ such that $m'' \in S \& Next(m,m'')$, then $m'$ and
$m''$ are both in $h$ and must be $\unlhd$-comparable. We cannot
have $m' \lhd m''$ since by $Next(m,m'')$ this would mean that $m'
\unlhd m$, in contradiction with our assumptions. Therefore, we
must have $m'' \unlhd m'$, whence by $R \subseteq R_e$, the future
always matters constraint, and Definition \ref{sigma}.2 we again
get that $m' \in S$. Thus we get $m' \in S$ anyway, which means
that $Act(m',h) = \{ x,y \}$ and the constraint is satisfied.

\textbf{Presenting a new proof makes histories divide}. Assume
that $h, g \in H_m$ and that there exists an $m' \rhd m$ such that
$m' \in g \cap h$. We consider four cases according to the above
definition of $Act$:

\emph{Case 1}. $m \in S$. Then clearly $Act(m, h) = Act(m, g) = \{
x,y \}$ and the constraint is satisfied.

\emph{Case 2}. For some $m'' \in h$ it is true that $m'' \in S$
and $Next(m, m'')$. Then $Act(m, h) =  \{ x,y \}$, and also $m'$
and $m''$ must be $\unlhd$-comparable.  We cannot have $m' \lhd
m''$ since by $Next(m,m'')$ this would mean that $m' \unlhd m$, in
contradiction with our assumptions. Therefore, we must have $m''
\unlhd m'$, whence by Lemma \ref{technical}.2, we must have $m''
\in g$ so that we get $Act(m, g) = \{ x,y \}$ as well, and the
constraint is satisfied. A symmetrical (and similar) subcase would
start from the assumption that $m'' \in g$.

\emph{Case 3}. $m = m_0$ and $h \approx_{m_0} h_2$. Then $Act(m,
h) =  \{ y \}$. By Lemma \ref{technical}.3, we get that $g
\approx_{m_0} h_2$ so that $Act(m, g) = \{ y \}$ as well, and the
constraint is satisfied. Again, a symmetrical (and similar)
subcase would start from the assumption that $g \approx_{m_0}
h_2$.

\emph{Case 4}. None of the above cases applies either for $h$ or
for $g$. Then $Act(m, h) = Act(m, g) = \emptyset$ and the
constraint is satisfied.

\textbf{No new proofs guaranteed}. Let $m \in Tree$ be arbitrary.
If $Act_m = \emptyset$, then the constraint is satisfied
trivially. On the other hand, if $Act_m \neq \emptyset$, then, by
Claim 3, we must have both $m \in S$ and $Act_m = \{ x,y \}$. Then
we have to consider two cases:

\emph{Case 1}. There exists an $m' \lhd m$ such that $m' \in S$.
Then choose an arbitrary $h \in H_m$. We have $m' \in h$ by Lemma
\ref{technical}.2, and $Act(m', h) = \{ x,y \}$ by the above
definition of $Act$, so that the constraint is satisfied.

\emph{Case 2}. For all $m' \lhd m$ we have $m' \notin S$. Then, by
Definition \ref{sigma}.4 we must have:
$$
(\exists m_2 \lhd m)(\forall m_3 \lhd m)(\neg m_2 \lhd m_3).
$$
We choose such an $m_2$. Of course, whenever $m_3 \lhd m$, $m_3$
must be $\unlhd$-comparable to $m_2$ by the absence of backward
branching, therefore, given that we never have $m_2 \lhd m_3$, we
must get that:
$$
(\forall m_3 \lhd m)(m_3 \unlhd m_2).
$$
Adding this up with $m_2 \lhd m$, we get that $Next(m_2,m)$. Now,
choose an arbitrary $h \in H_m$. We have $m_2 \in h$ by Lemma
\ref{technical}.2 and $Act(m_2, h) = \{ x,y \}$ by the fact that
$m \in h \cap S \& Next(m_2,m)$ and the above definition of $Act$,
so that the constraint is again satisfied.

\textbf{Presented proofs are epistemically transparent}. Assume
that $m, m' \in Tree$ are such that $R_e(m,m')$. Then, if $Act_m =
\emptyset$, the constraint is satisfied trivially. On the other
hand, if $Act_m \neq \emptyset$, then, by Claim 3, we must have $m
\in S$. But then, by Definition \ref{sigma}.2, we will also have
$m' \in S$, and, by Claim 3, $Act_{m'} = Act_m = \{ x,y \}$ so
that the constraint is again satisfied.

Therefore, the above-defined $\mathcal{M}$ is shown to be a jstit
model for $Ag$ and we obviously have that:
$$
\mathcal{M}, m_0, h_2 \not\models K(\Box Ex \vee \neg\Box Ey) \to
(Ex \vee \neg Ey).
$$
Indeed, whenever $m' \in Tree$, then, by Claim 2 above, we will
either have $Act_{m'} = \emptyset$ (and then $\mathcal{M}, m', g
\models \neg\Box Ey$ for all $g \in H_{m'}$), or  $Act_{m'} = \{
x,y \}$ (and then $\mathcal{M}, m', g \models \Box Ex$ for all $g
\in H_{m'}$). Therefore, it is clear that we have:
$$
\mathcal{M}, m_0, h_2 \models K(\Box Ex \vee \neg\Box Ey),
$$
and yet, on the other hand it is true that:
$$
\mathcal{M}, m_0, h_2 \models \neg Ex \wedge Ey.
$$
\end{proof}
The frame definability result for jstit frames is now also
straightforward:
\begin{theorem}\label{frame2}
Let $F  = \langle Tree, \unlhd, Choice, R, R_e \rangle$ be a jstit
frame for $Ag$. For any constant specification $\mathcal{CS}$ it
is true that:
$$
(\forall \mathcal{M} \in Mod_{\mathcal{CS}}(\{ F \}))(\mathcal{M}
\models \{ A \in Form^{Ag} \mid \vdash_{\mathcal{CS}} A \})
\Leftrightarrow F \in \mathcal{F}^{Ag}_{reg}.
$$
\end{theorem}
\begin{proof}
Same as for Theorem \ref{frame1}, using Theorem
\ref{j-completeness} and Lemma \ref{jstit-falsify} in place of
Theorem \ref{completeness} and Lemma \ref{stit-falsify},
respectively.
\end{proof}

\section{Conclusions and further research}\label{conclusion}

We have established that $\Sigma_D$, our axiomatization of stit
logic of justification announcements from Part I, has a reasonably
clear-cut meaning (given by condition \eqref{mixsucc} in
Definition \ref{classes-stit} above) when it comes to restrictions
induced by it on the temporal substructure of the underlying
frame. The fact that $\Sigma_D$, as it follows from the main
result of Part I, cannot distinguish between mixed successor
frames and a group of other stronger restrictions all the way up
to discrete time structures underscores the limitations of
expressive power of JA-STIT. We have also seen that once the
epistemic accessibility relations enter the picture, the
complexity of the restriction on frames imposed by $\Sigma_D$ goes
up significantly. One may even question the possible utility of
such a complex defining condition as an insight into the nature of
JA-STIT. We believe, however, that the notion of the family of
sets $\Theta_m$ for a given moment is interesting at least in that
the claims we have established in the course of proofs of Theorems
\ref{soundness} and \ref{frame2} given above apparently suggest
that these families allow one to more or less characterize, for a
given finite set of proof polynomials $\sigma$, the set of moments
$m$ in a given jstit model for which we have $\sigma \subseteq
Act_m$ without mentioning $Act$ at all. This by-product of the
above results probably holds some potential for further research
in this direction.

Additionally, the above results lay down a basis for similar
enquiries into the expressive powers of some natural extensions of
JA-STIT, like the logic of $E$-notions (see \cite{OLWA2}) or the
full basic jstit logic introduced in \cite{OLWA}.

\section{Acknowledgements}
To be inserted.

}

\end{document}